\documentclass[psamsfonts]{amsart}
\usepackage[utf8]{inputenc}
\usepackage{amsfonts}
\usepackage{xcolor}
\usepackage{hyperref}
\definecolor{linkcolor}{HTML}{5D8AA8} 
\definecolor{urlcolor}{HTML}{5D8AA8} 
\hypersetup{
pdftitle={Stable vector bundles on the families of curves},
pdfsubject={Mathematics, Algebraic geometry, Differential Geometry},
pdfauthor={Fedor Bogomolov and Lukzen Elena},
pdfkeywords={Stability, Vector bundle,Mathematics, arXiv, Preprint}
 linkcolor=linkcolor,urlcolor=urlcolor, colorlinks=true
}
\usepackage{amsmath}
\usepackage{xcolor}
\usepackage{amsthm}
\usepackage{pdflscape}
\usepackage{pgfplots}
\usepackage{mathrsfs}
\usepackage{ textcomp }
\newtheorem{thm}{Theorem}[section]
\newtheorem{cor}[thm]{Corollary}

\theoremstyle{definition}
\newtheorem{defn}[thm]{Definition}

\theoremstyle{remark}
\newtheorem{rem}[thm]{Remark}

\makeatletter
\let\c@equation\c@thm
\makeatother
\numberwithin{equation}{section}

\title{Stable vector bundles on the families of curves}

\begin{document}

\begin{abstract}
We offer a new approach to proving the Chen-Donaldson-Sun theorem which we demonstrate with a series of examples. We discuss the existence of a construction of a special metric on stable vector bundles over the surfaces formed by a families of curves and its relation to the one-dimensional cycles in the moduli space
 of stable bundles on curves.


\end{abstract}
\author{Fedor Bogomolov}
\address{Fedor Bogomolov \\Courant Institute, New York Universuty, 251 Mercer Str., New York, USA 10012. Also: Laboratory of Algebraic Geometry, National Research University HSE, Department of Mathematics, 6 Usacheva Str. Moscow, Russia}
\curraddr{}
\email{bogomolo@cims.nyu.edu}
\thanks{}

\author{Elena Lukzen}
\address{Elena Lukzen \\ SISSA, via Bonomea 265, 34136 Trieste, Italy}
\curraddr{}
\email{elukzen@sissa.it}
\thanks{}

\maketitle
\tableofcontents

\section{Introduction}

Stability of vector bundles on different geometric spaces has been an object of study for a long time, deriving from work in algebraic geometry. The concept of stability arose naturally from the question of whether a space which could parametrize all bundles of a given rank exists, moreover, a subsequent question was if there is a special class of vector bundles which could be naturally parametrized in this way.  Significant progress in the direction has been made by Mumford, Narasimhan-Seshadri \cite{NS}, Maruyama and others.
 We explore how one-dimensional cycles in the moduli space of vector bundles on curves correspond to a bundles on families of curves which are stable on the restriction to every curve of the family. 
We present a connection between existence of a special metric on a class of surfaces formed by the families of curves and its natural relation to one-dimensional cycles in the moduli space of stable bundles on curves.
Let $B$ be a family of curves in the Kodaira sense,  $ B= \{ C_{\lambda} \}_{\lambda \in \Lambda}$  , where $\Lambda$ is a parameter set. Let $\mathcal{E}$ be a vector bundle on $B$.
If we can find a criterion for stability on $\mathcal{E}$ it  will be possible to find  corresponding criterion for vector bundles on  algebraic projective surfaces, by  covering them with a suitable family of curves.\\
Thus the first part is devoted to the definition and analyses of a special geometric cycle which corresponds to every vector bundle on the families of curves. Secondly, we construct a method of its interpretation algebraically via unitary representations. Thirdly, we demonstrate how to analyze the notion of stability of a vector bundle on the complex projective algebraic surface.  
The scheme of the proof has  following steps:

Find a description of vector bundles on the families of curves.
\begin{defn}
\label{def1}
 Every vector bundle $\mathcal{E}$ on $B$, which is stable on the restriction to every non-singular curve and  a curve with relatively small number of singularities $E|_{C_\lambda}$ from our family $B=\{C_{\lambda}\}$ is a smooth \textit{moduli section $\mathscr{S} \in \Gamma(M_{\lambda})_{\lambda \in \Lambda}$} of the family of moduli spaces ${M}_{C_\lambda}$, $\lambda \in \Lambda$.
\end{defn}

\begin{rem}
\label{rem1}
If $C_{\lambda}=C$ is a constant curve, then all moduli spaces are the same ${M}_{C_{\lambda}}=M_C=M$, thus a vector bundle on $B$ is encoded by  a \textit{moduli section} $\mathscr{S} \subset M$ in the moduli space $M$:
$$\mathscr{S}= { m_\lambda \in {M}, \lambda \in \Lambda } $$ 
\end{rem}

\begin{rem}
\label{rem2}
The moduli space $M_{B}$ of all vector bundles on $B$, whose restriction to any curve of the family is stable,  consists of all \textit{moduli sections} $\mathscr{S}$ of families of the corresponding moduli spaces  ${M}_{C_\lambda}$.\\
$$
M_B=\{ \mathscr{S} | \mathscr{S} \in \Gamma( ( M_{C_\lambda} )_{\lambda \in \Lambda}) \} 
$$
\end{rem}

\begin{rem}
\label{rem3}
If $C_{\lambda}=C_1$ and $\Lambda=C_2$, thus we have a vector bundle on the $C_1 \times C_2$. Then every subset  $\mathscr{S} \subset M$  will form a \textit{ cycle} in the moduli spaces of vector bundles $\mathscr{S}=\mathscr{C}_2 \subset M_{C_1}$ or $\mathscr{C}_1 \subset M_{C_2}$.

\end{rem}

Actually a more general notion can be defined:

\begin{defn}
Every vector bundle $\mathcal{V}$ on the family of algebraic varieties $\{X_{\lambda}\}_{\lambda \in \Lambda}$, which is stable on the restriction to every $|_{X_\lambda}$  is a smooth \textit{moduli section $\mathscr{S}$} of the family of Moduli spaces $\{{M}_{X_\lambda}\}$, $\lambda \in \Lambda$.
\end{defn}

Or
\begin{defn} \textbf{(General principle).}
Let $\{Q_{\lambda}\}$ be a family of sets, which parametrizes objects  on a family of spaces $ \mathcal{B}=X_{\lambda}$. Then a set, which parametrizes objects on $\mathcal{B}$ is $\Gamma(\{ Q_{\lambda}\}_{\lambda \in \Lambda})$.

\end{defn}

The main theorem of the article is a Theorem (\ref{omega}):
\begin{thm}
Let $E \to \mathcal C$ a vector bundle over a one-dimensional family of curves then  

\begin{align}
c_2(E)=\sum_i (\Omega^{4})_{ii}=
\sum_{i,j} det(\partial^{2}h_{ij})  dz^{1} \wedge d\bar{z}^{1} \wedge dz^{2}  \wedge d\bar{z}^{2}
\end{align}
and consequently $c_2(E)>0$,

 where $h_{ij}$ is a metric on a vector bundle $E$.
\end{thm}

The paper is based on two talks given at SISSA held on 15 May 2019 ( Definition [\ref{def1}] and Remarks [\ref{rem1},\ref{rem2},\ref{rem3}]) and on 27 May 2020. \\
The paper is organized as follows. In Section 2 we review preliminary facts on restrictions theorems in algebraic geometry, the main fact there is the Bogomolov's restriction theorem. In Section 3 we define a \textit{moduli section (or cycle)} which describes the set of bundles which are stable on families of curves. In Section 4 we study curvature properties of vector budnles on the $C_1 \times C_2$. Sections 5, 6, 7 are devoted to a computation of the curvature and the second chern class for a vector bundles on $C_1 \times C_2$ and $\mathcal{C} \to C$ one-dimensional families of curves. In Section 9 we indicate future directions of research, i.e. stability of vector bundles on algebraic surfaces. Section 10 is devoted to the announcement of one of the future projects dedicated to the study of the Calabi-Yau metric on vector bundles over surfaces.

\textbf{Acknowledgements.} Fedor Bogomolov is partially supported by the HSE University Basic
Research Program, Russian Academic Excellence Project ’5-100’
and by
EPSRC programme grant EP/M024830. Elena Lukzen is supported by the SISSA PhD fellowship and  acknowledges the support from the European Union’s H2020 research and innovation program under the Marie Sk\l{}odowska–Curie grant No. 778010 \textit{IPaDEGAN}.

E.L. would like to thank all the members of Geometry group in SISSA for a friendly and stimulating international research  environment. E.L. is especially grateful to Jacopo Stoppa for many help and encouragement during her studies. E.L. expresses her gratitude to Courant Institute of Mathematical Sciences for the kind hospitality during her visits. 

\section{Preliminary facts and background}

Let us recall a number of facts following the exposition presented in the book \cite{HL}.
In the case of smooth projective surfaces we have a non-degenerate hyperbolic scalar product on the group $Pic_{R} V$ and duality between the cone of polarizations and the cone of effective divisors.

Let $Num$ denote the free Z-module $Pic(X)$ factorized by the numerical equivalence. The intersection product defines an integral quadratic form on $Num$.\\
The Hodge Index Theorem says, that, over
$\mathbf R$, the positive definite part is $1$-dimensional.\\
In other words, $Num_{\mathbf R}$
carries the Minkowski metric. For any class $u \in Num_{R}$ let $|u|=|u^{2}|^{1/2}$. 
Denote as $A$ a subcone consisting of all ample divisors.  A polarization of $X$
is a ray $R_{>0}.H$, where $H \in A$. Let
$\mathscr{H}$ denote the set of rays in $K_{+}$. This set can be identified with the hyperbolic manifold
$ \{ H \in K_{+}: |H|=1 \}$.\\The hyperbolic metric $\beta$ is defined as follows:\\
for points
$[H],  [H'] \in \mathcal{H}$ let
$$
\beta ([H],[H'])=arccosh ( \frac{H.H'}{|H|.|H'|})
$$
Denote the open cone
$
K^{+}=\{D \in Num_{R} | D^{2}>0, D.H>0$ for all ample divisors H $ \}$.
Note that the second condition is added only to pick one of the two connected components of the set of all $D$ with $D^{2}> 0$. This cone contains the cone of ample divisors and in turn is contained in the cone of effective divisors.
$K^{+}$ satisfies the following property:\\
$
D \in K^{+} \iff D.L>0$ for all $L \in \overline{K^{+}}\setminus \{0\}$\\
For any pair of sheaves $G, G'$ with nonzero rank let
$$
\xi_{G', G}:=\{\frac{c_1(G')}{rk(G')}- \frac{c_1(G)}{rk(G)}\} \in Num_{R}
$$
This invariant in the case of a destabilizing sheaf plays role as "destabilizing" point in the corresponding cone.
\begin{thm}
Let $X$ be a smooth projective variety of dimension $n$ and $H$ an ample divisor on $X$. If $F$ is a
semistable torsion free sheaf, then
$$
\delta(F).H^{n-1} \geq 0
$$
\end{thm}

\begin{thm}(Bogomolov's effective restriction \cite{B1,B3})
Let $W$ be a locally free sheaf of rank $r \geq 2$ on a family of curves with $c_1(W)=0$. Assume that $W$ is $\mu$-stable with
respect to an ample class $H \in K^{+} \cap Num$ and $C\subset X$ be a smooth curve with
$[C] = nH$. Let $2n \geq \frac{R}{r} \delta(F) + 1$. Then $F|_C$ is a stable sheaf.
\end{thm}

The consequence of the thm is that depending on $c_1$, rank the restriction of a vector bundle on curves in particular on those with a relatively small number of singularities, is stable.

\section{Moduli Section}

Let us introduce a new object, arising from every vector bundle on particular families.
Suppose we have $B$ a family of curves and $\mathcal{E}$ a vector bundle on it.
We can say that $\mathcal{E} |_{C_{\lambda}}$ restricted to a point of the family, i.e. a curve $C_{\lambda}$,  is a bundle on a Riemann surface with the corresponding unitary connection $A$.
The information of stability of $\mathcal{E}$ is contained in its second chern class  $c_2(\mathcal{E})$.

\begin{defn}
\label{moduli section}
Every vector bundle $\mathcal{E}$ on $B$, which is stable on the restriction to any curve of the family $B$, is a smooth section of the family of moduli spaces ${M}_{C_\lambda}$, $\lambda \in \Lambda$.
 $$
 \mathcal{E} \in \Gamma({M}_{C_\lambda})_{\lambda \in \Lambda}
 $$
\end{defn}

\textit{Motivation to definition.}
Vector bundle $\mathcal{E}$ on the family of curves $B$ on the restriction to the point $\mathcal{E}|_{C_{\lambda_0}}$, i.e. a particular curve $C_{\lambda_0}$, is a vector bundle , arising from the representation of the fundamental group of $\pi_1(C_{\lambda_0})$ (by the Narasimhan-Seshadri thm). Therefore we can simply consider the following fibration on $B$: the fiber on every curve $C_{\lambda}$ corresponds to its moduli space ${M}_{C_\lambda}$. Consequently, to "fix a vector bundle on the family $B$ " simply means that on each fiber, i.e. on the corresponding moduli space, we have to choose one element. It identically repeats the definition.

To supplement our definition we should note that stable bundles come from
\begin{itemize}
    \item Stable bundles on normalizations of the singular curves: $\mathcal{E} \to \hat C$, where $ \nu: \hat C \to C $ is a normalization. These bundles one can get from representations of a fundamental group of the normalization $\pi_1(\hat C)$ and thus the destabilizing sheaf is not induced on the initial curve $C$.
\item Bundles which are unstable on the normalization $\hat C$, but are stable on the singular curve $C$ (parabolic semistable vector bundles).
They appear in the following way: we can take an unstable bundle $\hat{\mathcal{E}}$ (i.e. with destabilizing subsheaf $\mathscr{F}$) on the normalization $\hat C$. When we arrive to a base $C$ it transforms to a point $\mathcal{E}^{sing}$.
In such way we obtain the identification of the parameters of the destabilizing sheaf $\mathscr{F}$  with the structure of the parabolic bundle on $C$.
It means that on $\hat{\mathcal{E}} \to \hat{X}$, where $\hat{\mathcal{E}}$ is an unstable bundle at the point on the normalization $\hat X$ ( $\hat X$ denotes $X$ blown-up in a a node of a singular curve $C$, $\nu: \hat C \to C$). Then it has a normal destabilizing subsheaf $\mathscr{F} \subset \hat{\mathcal{E}}$ such that $c_1(\mathscr{F})-(\frac{rk\mathscr{F}}{k})c_1(\mathcal{E}) \in K^{+}$.\\
A blow-down map $\hat \nu$ induces a morphism: 
$$
\hat{\nu}: \mathscr{F} \to \mathcal{E}_{p},
$$
where $\mathcal{E}_{p}$ is a flag variety of type determined by a fixed quasiparabolic structure i.e. a flag $\mathcal{E}_{p}=F_{1}\mathcal{E}_{p} \supset F_{2}\mathcal{E}_{p} \supset... \supset F_{r}\mathcal{E}_{p}$ and weights $\alpha_1,..,\alpha_r$ attached to $F_1\mathcal{E}_p,..,F_r\mathcal{E}_P$  such that $0 \leq \alpha_1\textless \alpha_2 \textless ...\textless \alpha_r \textless 1$, $k_{1}=dimF_{1}\mathcal{E}_p-dimF_{2}\mathcal{E}_p,  k_{r}=dimF_{r}\mathcal{E}_P$  the multiplicities of $\alpha_1,.., \alpha_r$.

These are parabolic bundles.
Thus we get a strata, which is glued to our families of moduli spaces of vector bundles.
It is known that the set of all parabolic semi-stable bundles is an open subset of a suitable Hilbert scheme, which has the usual properties, i.e. non-singular, irreducible and of a given dimension. This open set could be mapped to a product of Grassmannians and Flag varieties.

\end{itemize}

\begin{rem}

If $C_{\lambda}=C$ is a constant curve, then all the moduli spaces are the same ${M}_{C_\lambda}={M}_C={M}$, thus every vector bundle on $B$ is decoded by  the subset $\mathscr{S} \subset M$ in the moduli space $M$:
$$\mathscr{S}= < m_\lambda \in {M}, \lambda \in \Lambda > $$ 
\end{rem}

\begin{rem}

Moduli space ${M}_{B}$ of all the vector bundles on $B$  consist of all the sections $S$ of family of the corresponding moduli spaces  ${M}_{C_\lambda}$.

\end{rem}

\begin{rem}

If $C_{\lambda}=C_1$ and $\Lambda=C_2$, we have a vector bundle on the $C_1 \times C_2$. Then every subset  $\mathscr{C}_1 \subset M_{C_2}$ or $\mathscr{C}_2 \subset M_{C_1}$ will form a cycle in the corresponding moduli space of vector bundles.

\end{rem}
This way we have a correspondence between families of stable vector bundles and the sections of families of moduli spaces of vector bundles over a given base.

\section{Vector bundles on the product of curves and the curvature}

Let $\mathcal{E} \to C_1 \times C_2$ be a vector bundle such that $\mathcal{E}|_{C_1}$ and $\mathcal{E}|_{C_2}$ stable.
Coordinates on $\mathcal{E} \to C_1 \times C_2$ could be described as a pair of representations $(\tau, \rho)$, where $\tau$ corresponds to a horizontal bundle and
$\rho$ corresponds to a vertical bundle.
The fiber is equal to $\mathcal{E}|_{(x,y)}= {E}_{\rho}\cap {E}_{\tau}= {E}|_{(\rho,\tau)}$

As is well-known that along $E_{\rho}$, $E_{\tau}$ exist Ehresmann connections which can be viewed as  horizontal subbundles $H_{\rho}$,$H_{\tau}$: $$
T{E}_{\rho}=H_{\rho}\oplus V_{\rho}, 
T{E}_{\tau}=H_{\tau}\oplus V_{\tau}.
$$
Thus we have two families of Ehresmann connections: 
$
\{H_{\rho}, \rho \in \mathscr{C}_1 \}
$,
$
\{H_{\tau}, \tau \in \mathscr{C}_2 \},
$
where $\mathscr{C}_1,\mathscr{C}_2$ are the moduli cycles.\\
We are seeking for the connection $H$ on the bundle $\mathcal{E}$ such that $T{\mathcal{E}}=H \oplus V$, where
 $V$ is a tangent space to a fiber. Note that
$
V= T({E}_{\rho}\cap {E}_{\tau})=(H_{\rho}\oplus V_{\rho}) \cap (H_{\tau}\oplus V_{\tau}).
$\\
The tangent to $T\mathcal{E}$ is a tangent bundle of a moduli cycle $\mathscr {C}_1$.\\
Locally $T\mathcal{E}= TE_{\rho} \oplus TE_{\tau} $, therefore we can define a possible connection at least set-theoretically:\\ 
\begin{align}
\label{conn}
H=(H_{\rho}\oplus V_{\rho}) \oplus (H_{\tau}\oplus V_{\tau}) \textfractionsolidus \{(H_{\rho}\oplus V_{\rho}) \cap (H_{\tau}\oplus V_{\tau})\}
\end{align}

\begin{thm}

Denote $\Theta_{E}$ a curvature of connection $\nabla_E$. The obstruction to the Jordan property - failure for the tangent fields $V,W$ to connection to generate a Lie algebra- is exactly the curvature $\Theta_{\mathcal{E}}$.

\end{thm}
\begin{proof}
Our goal, in general, is to find an expression for the curvature tensor $\Theta$ of the connection on the vector bundle $\mathcal{E}$. Therefore we would be able to express the second Chern class $c_2(\mathcal{E})$ explicitly, depending on a moduli cycle $\mathscr{S}$: 
$$c_2(\mathcal{E}) = \frac{tr(\Theta(\mathscr{S})^2)-tr^2(\Theta(\mathscr{S}))}{ 8\pi^{2}}$$

Notice that there exist only two "directions" on the base space $C_1 \times C_2$: a direction towards a curve $C_1$ and a direction towards a curve $C_2$.
When we lift vector fields from the base $C_1 \times C_2$ to a vector bundle, we get the vector fields $V,W$ along the "cycle/section" $\mathscr{ C}_1$ and along the curve $C_2$. $V$ and $W$ induce diffeomorphisms on a vector bundle $\mathcal{E}$, which act on particular sections $s \in \Gamma(\mathcal{E})$. If we differentiate the result twice, we will get $\Theta(s)$ a value of curvature on $s$. If we take a corresponding basis of sections of $\mathcal{E}$ and lift the basis vector fields from the base to our cycle/section $S$ we would be able to count $\Theta$ explicitly.
As is well know in this case, the obstruction to Jordan property - failure for the tangent fields $V,W$ to connection to generate a Lie algebra- is the curvature $\Theta_{\mathcal{E}}$.
\end{proof}

Note that we can use Gauss-Manin connection which acts as $\nabla_{GM}: T\mathscr{C} \to V$ (by the reason that $\nabla_{GM}$ is flat), where $V$, as above, is a tangent space to a fiber.

Recall that by non-abelian Hodge theory, the tangent space at the point $(C, U)$ which is equal to $M_{dR}(C_u)$ is canonically identified with $H^{1}_{dR}(C_u, End E)$.

Therefore a Gauss-Manin connection computes the derivative of our section $S$.
Notice that we have not only Gauss-Manin connection but also the other one induced from the connections from the both directions whose properties we will study in the following sections.

\section{Curvature of $E \to C_1\times C_2$}
Let $E \to C_1 \times C_2$ is a bundle , which is stable on $C_1 \times\{ x\}$ and $C_2 \times \{ y\}$ for any $x, y \in C_1,C_2$.
\subsection {rank $E=2$}

In the case $E \to C_1\times C_2$ and the bundle is of rank $2$
we consider a connection defined by unitary flat connections in both
directions (which we can define at least set-theoretically (\ref{conn})).\\
We can take a basis of sections which is given by the basis of sections on $E|_{C_1}$ which we denote as $(s_1,s_2)$ and $E|_{C_2}$ $(s'_1,s'_2)$, thus a local basis of sections on $E$  is $(s_1,s_2,s'_1,s'_2)$.

Then we consider a curvature matrix $R$ of the second derivatives $d^2 h_{i,k}/dz_i dz_k$. As the curvature vanishes on $E|_{C_1} $ and on $E|_{C_2}$ the corresponding second derivatives of $h_{12}, h_{11},h_{21},h_{22}$ and $h'_{12}, h'_{11},h'_{21},h'_{22}$ vanish. Then by the hermitian condition
in our $4\times 4 $ matrix the $2\times 2$ diagonal squares are zero and 
antidiagonal matrices
are conjugated (\ref{partmatrix}).\\

We will act by an operator matrix which is $4\times 4$ matrix on each $h_{ij}$ ,\\ schematically depicted as: 
$$
\partial^2= \begin{pmatrix}
\partial z_1\partial z_1 & \partial z_1\partial \bar{z_1}  & \partial z_1\partial z_2 & \partial z_1\partial \bar{z_2} \\
\partial\bar{ z_1}\partial z_1 & \partial\bar{ z_1}\partial \bar{z_1}  & \partial \bar{z_1} \partial z_2 & \partial \bar{z_1}\partial \bar{z_2}  \\
\partial z_2\partial z_1 & \partial z_2\partial \bar{z_1}  & \partial z_2\partial z_2 & \partial z_2\partial \bar{z_2}  \\
\partial\bar{ z_2}\partial z_1 & \partial\bar{ z_2}\partial \bar{z_1}  & \partial \bar{z_2} \partial z_2 & \partial \bar{z_2}\partial \bar{z_2} 
\end{pmatrix}
$$

\begin{equation}
    \label{partmatrix}
 \partial^2(h_{ij}) =\begin{pmatrix}
0 & 0  & \bar{a_{i,i}} & \bar{a_{i,j}} \\
0 & 0  & \bar{a_{j,i}} & \bar{a_{j,j}} \\
a_{i,i} & a_{i,j} &  0 & 0 \\
a_{i,j} & a_{j,j} & 0 & 0
\end{pmatrix}
\end{equation}

\subsection{rank $E > 2$}
 For a higher rank bundles we have the same situation for some $2$-minors.
The curvature matrix is expressed through the second derivatives
and so the above inequality should also tell us about the
class $c_2$.

As the basis of a local frame field is $s_{U}=(s_1,..,s_r,s'_1,..,s'_l)$ then analogously on $E|_{C_1}$ and $E|_{C_2}$ the induced connection $D$ is flat. Therefore the curvature vanishes and the second derivatives of $h_{11},h_{12}..h_{r1}..,h_{rr}$  and $h'_{11},h'_{12}..h'_{r1}..,h'_{rr}$ are zero.
Only in the mixed directions (i.e. $s_i,s'_j$) the mixed derivatives would not vanish, i.e. $\partial_{i}\partial_{j}$, $\partial_{i}\bar{\partial}_{j}$ etc.\\

The action of the connection is\\ 
$
Ds_1=0s_1+..+0s_r+ a_{11}s'_1+..+a_{1r}s'_r\\
Ds_2=0s_1+..+0s_r +a_{21}s'_1+..+a_{2r}s'_r\\
...\\
Ds'_1= a'_{11}s_1+..+a'_{1r}s_r+ 0s'_1+..+0s'_r\\
Ds'_2= a'_{21}s_1+..+a'_{2r}s_r+ 0s'_1+..+0s'_r\\
...
$

Let us consider the coordinates on the base $C_1 \times C_2$: $(z_1,\bar{z}_1,z_2,\bar{z}_2)$. The first part of coordinates correspond to the coordinates on a curve $C_1$ and the second on a curve $C_2$.\\
Thus we have $s_{U}=(s_1,..,s_r,s'_1,..,s'_r)$. Again denote $h_{ij}=h(s_i,s_j)$, $h'_{ij}=h(s_i,s'_j)$,$h''_{ij}=h(s'_i,s'_j)$.\\
Notice that as our chosen connection is flat in both directions \\
$\partial^{2} h_{ij}, \partial^{2} h'_{ij}  , \partial^{2} h''_{ij} $ will all have a form [\ref{partmatrix}].

Again we will act on each $h_{ij}$ by an operator matrix which is $4\times 4$ matrix,\\ schematically : 
$$
\partial^2= \begin{pmatrix}
\partial z_1\partial z_1 & \partial z_1\partial \bar{z_1}  & \partial z_1\partial z_2 & \partial z_1\partial \bar{z_2} \\
\partial\bar{ z_1}\partial z_1 & \partial\bar{ z_1}\partial \bar{z_1}  & \partial \bar{z_1} \partial z_2 & \partial \bar{z_1}\partial \bar{z_2}  \\
\partial z_2\partial z_1 & \partial z_2\partial \bar{z_1}  & \partial z_2\partial z_2 & \partial z_2\partial \bar{z_2}  \\
\partial\bar{ z_2}\partial z_1 & \partial\bar{ z_2}\partial \bar{z_1}  & \partial \bar{z_2} \partial z_2 & \partial \bar{z_2}\partial \bar{z_2} 
\end{pmatrix}
$$

  \section{Calculation of the Second Chern class}

Recall that the Dolbeaut representative of the Atiyah class is given by $[\bar \partial, \Omega]$, where $\Omega$ is a curvature tensor on the vector bundle $E$.
Recall that $ \Omega^{i}_{j}=\sum R^{i}_{j\alpha\bar{\beta}} dz^{\alpha}\wedge d \bar{z}^{\beta}$,
therefore, for $ C=\Omega \wedge \Omega = \sum_{k}(\Omega^{i}_{k} \wedge  \Omega^{k}_{j})$. As every $\Omega_{ij}$ has a form $\Omega^{i}_{j}= \sum R^{i}_{j\alpha\bar{\beta}} dz^{\alpha}\wedge d \bar{z}^{\beta}$ .

\begin{align*}
C_{ij}= &\sum \Omega^{i}_{k} \wedge \Omega^{k}_{j}= \sum R^{i}_{k\alpha\bar{\beta}} dz^{\alpha}\wedge d \bar{z}^{\beta}  \wedge \sum R^{k}_{j\alpha\bar{\beta}} dz^{\alpha}\wedge d \bar{z}^{\beta} =\\
&=\sum_k \sum_{\alpha, \beta,\gamma,\theta} \sum_{\sigma \in S_4} (-1)^{\sigma} R^{i}_{k\sigma(\alpha \bar{\beta})}
R^{k}_{j\sigma(\gamma \bar{\theta})}
dz^{\alpha} \wedge dz^{\beta} \wedge d\bar{z}^{\gamma} \wedge d\bar{z}^{\theta}
\end{align*}

Denote $({\alpha})=(\alpha,\beta,\gamma,\theta)$ a multindex consisting of 4 variables.
Then, 
\begin{align*}
(\Omega^4)_{ii}&= \sum_j (\Omega^2)_{ij} \wedge (\Omega^2)_{ji}=\sum_{j} \sum_{(\alpha)} [\sum_k \sum_{\sigma \in S_4} (-1)^{\sigma} R^{i}_{k\sigma'*\sigma(\alpha \bar{\beta})}
R^{k}_{j\sigma'*\sigma(\gamma \bar{\theta})}
]\cdot\\
&[\sum_k  \sum_{\sigma \in S_4} (-1)^{\sigma} R^{j}_{k\sigma'*\sigma(\alpha' \bar{\beta'})}
R^{k}_{i\sigma'*\sigma(\gamma' \bar{\theta'})}
] dz^{\alpha} \wedge dz^{\beta} \wedge d\bar{z}^{\gamma} \wedge d\bar{z}^{\theta}   
\end{align*}

Recall that $$R_{j\bar{k}\alpha \bar{\beta}}=\sum h_{i\bar{k}}R^{i}_{j\alpha\bar{\beta}}=-\partial_{\bar{\beta}}\partial_{\alpha}h_{j\bar{k}}+ \sum h^{a\bar{b}}\partial_{\alpha}h_{j\bar{b}}\partial_{\bar{\beta}}h_{a\bar{k}}$$,

For our expression we would need the expressions of this kind:
$$
R^{j}_{k \sigma'*\sigma(\alpha \bar{ \beta})}= -\partial_{\sigma'*\sigma(\bar{\beta})}\partial_{\sigma'*\sigma(\alpha)}h_{i\bar{k}}+ \sum h^{a\bar{b}}\partial_{\sigma'*\sigma(\alpha)}h_{i\bar{b}}\partial_{\sigma'*\sigma{\beta}}h_{a\bar{k}}
$$
  
Then we will have:
\begin{align*}
\label{expr}
 (\Omega^4)_{ii}&= \sum_{j} \sum_{\alpha} [\sum_k \sum_{\sigma \in S_4} (-1)^{\sigma} 
(-\partial_{\sigma(\bar{\beta})}\partial_{\sigma(\alpha)}h_{i\bar{k}}+ \sum h^{a\bar{b}}\partial_{\sigma(\alpha)}h_{i\bar{b}}\partial_{\sigma{\beta}}h_{a\bar{k}}) \cdot\\
&\cdot(-\partial_{\sigma(\bar{\theta})}\partial_{\sigma(\gamma)}h_{k\bar{j}}+ \sum h^{a\bar{b}}\partial_{\sigma(\gamma)}h_{k\bar{b}}\partial_{\sigma{\theta}}h_{a\bar{j}})
]\\
&\cdot[\sum_k  \sum_{\sigma \in S_4} (-1)^{\sigma}(-\partial_{\sigma(\bar{\beta})}\partial_{\sigma(\alpha)}h_{j\bar{k}}+ \sum h^{a\bar{b}}\partial_{\sigma(\alpha)}h_{j\bar{b}}\partial_{\sigma{\beta}}h_{a\bar{k}})\cdot\\
&\cdot(-\partial_{\sigma(\bar{\theta})}\partial_{\sigma(\gamma)}h_{k\bar{i}}+ \sum h^{a\bar{b}}\partial_{\sigma(\gamma)}h_{k\bar{b}}\partial_{\sigma{\theta}}h_{a\bar{i}})] dz^{\alpha} \wedge dz^{\beta} \wedge d\bar{z}^{\gamma} \wedge d\bar{z}^{\theta}
\end{align*}

Recall that the first derivatives  are 0 for a vector bundle on the product of the curves $C_1 \times C_2$, because we have a flat
unitary connection along
both directions. Therefore, only the second derivatives $\partial^2$ are non-trivial.
Therefore we will get an expression 
\begin{align*}
\sum_i (\Omega^{4})_{ii}&=\sum_{i,j,\alpha,\sigma \in S_4} \partial_{\sigma(\bar{\beta})}\partial_{\sigma(\alpha)}h_{i\bar{k}}\cdot \partial_{\sigma(\bar{\beta})}\partial_{\sigma(\alpha)}h_{j\bar{k}}\cdot \partial_{\sigma(\bar{\theta})}\partial_{\sigma(\gamma)}h_{k\bar{j}} \cdot \partial_{\sigma(\bar{\beta})}\partial_{\sigma(\alpha)}h_{j\bar{k}}  dz^{\alpha} \wedge dz^{\beta} \wedge d\bar{z}^{\gamma} \wedge d\bar{z}^{\theta}=\\
&=\sum_{i,j} det(\partial^{2}h_{ij})  dz^{1} \wedge d\bar{z}^{1} \wedge dz^{2}  \wedge d\bar{z}^{2}
\end{align*}

\begin{thm}
The resulting formula is

\label{omega}
\begin{equation}
\sum_i (\Omega^{4})_{ii}=
\sum_{i,j} det(\partial^{2}h_{ij})  dz^{1} \wedge d\bar{z}^{1} \wedge dz^{2}  \wedge d\bar{z}^{2}   
\end{equation}

\end{thm}

\section{Curvature of $E \to \mathcal{C} $}
Let $\varphi: \mathcal{C} \to C$ be a fibration of curves over a curve. Then we can again form an induced connection $D$ bearing it from one  $E|_C$ and every curve $E|_{\varphi^{-1} (c)}, c \in \mathcal{C}$.\\
Thus we can differentiate a section of a bundle $E$ using connection $D$ induced from flat connections from both directions. In this case the action of a connection $D$ along $E|_{C}$ will be zero as in the previous case, but the action of $D$ along the other direction may not be zero. 
We can again consider a matrix of the second derivatives of the metric $h_{i\bar{j}}$.
Since the basis of a local frame field $s_{U}=(s_1,..,s_r,s'_1,..,s'_l)$, where $s=(s_1,..,s_r)$ is a frame field along $E|_{\varphi^{-1} (c)}  $
and $s'=(s'_1,..,s'_l)$ is a local frame field along $E|_{C}$.
Thus
$$
D(s)= \nabla_{\varphi^{-1} (c)}(s)+\nabla_{C}(s')
$$

Therefore when we form a corresponding matrix of an action of connection $D$ on the local frame section.
It is if we will write $D(s)=(\nabla_{\varphi^{-1} (c)}+\nabla_{C})(s_1,..,s_r,s'_1,..,s'_r) .$

Let us consider the coordinates: $(z_1,\bar{z}_1,z_2,\bar{z}_2)$. The first part of coordinates correspond to coordinates on a curve $C$ and the second on the curve $\varphi^{-1}(c)$.\\
Thus we have $s_{U}=(s_1,..,s_r,s'_1,..,s'_r)$. Again denote $h_{ij}=h(s_i,s_j)$, $h'_{ij}=h(s_i,s'_j)$,$h''_{ij}=h(s'_i,s'_j)$.\\
Notice that \\
$\partial^{2} h_{ij} , \partial^{2} h'_{ij} ,\partial^{2} h''_{ij} $ vanish along $(z_1,\bar z_1)$ and nowhere else.

Again we will act on each $h_{ij}$ by an operator matrix which is $4\times 4$ matrix,\\ schematically : 
$$
\partial^2= \begin{pmatrix}
\partial z_1\partial z_1 & \partial z_1\partial \bar{z_1}  & \partial z_1\partial z_2 & \partial z_1\partial \bar{z_2} \\
\partial\bar{ z_1}\partial z_1 & \partial\bar{ z_1}\partial \bar{z_1}  & \partial \bar{z_1} \partial z_2 & \partial \bar{z_1}\partial \bar{z_2}  \\
\partial z_2\partial z_1 & \partial z_2\partial \bar{z_1}  & \partial z_2\partial z_2 & \partial z_2\partial \bar{z_2}  \\
\partial\bar{ z_2}\partial z_1 & \partial\bar{ z_2}\partial \bar{z_1}  & \partial \bar{z_2} \partial z_2 & \partial \bar{z_2}\partial \bar{z_2} 
\end{pmatrix}
$$

and

\begin{equation}
\label{metric}
 \partial^2(h_{ij}) =\begin{pmatrix}
0 & 0  & \bar{a_{i,i}} & \bar{a_{i,j}} \\
0 & 0  & \bar{a_{j,i}} & \bar{a_{j,j}} \\
a_{i,i} & a_{i,j} &  * & * \\
a_{i,j} & a_{j,j} & * & *
\end{pmatrix}
\end{equation}

The same expression for the Second Chern class is in the case of $\mathcal{C} \to C$ for a fibration for a curve over a curve, because the other entries of the corresponding matrix doesn't influence on the exression for the determinant.
(there is only a plane along one family, then there they appear, but
only by
one direction). We will show it in the next section.

\section{Second chern class for a $E \to \mathcal{C}$}

To analyze a formula for $\Omega^4$ we have to understand how do the first derivatives $\partial h_{ij}$ involved in the expression
\begin{align*}
 (\Omega^4)_{ii}&= \sum_{j} \sum_{\alpha} [\sum_k \sum_{\sigma \in S_4} (-1)^{\sigma} 
(-\partial_{\sigma(\bar{\beta})}\partial_{\sigma(\alpha)}h_{i\bar{k}}+ \sum h^{a\bar{b}}\partial_{\sigma(\alpha)}h_{i\bar{b}}\partial_{\sigma{\beta}}h_{a\bar{k}}) \cdot\\
&\cdot(-\partial_{\sigma(\bar{\theta})}\partial_{\sigma(\gamma)}h_{k\bar{j}}+ \sum h^{a\bar{b}}\partial_{\sigma(\gamma)}h_{k\bar{b}}\partial_{\sigma{\theta}}h_{a\bar{j}})
]\\
&\cdot[\sum_k  \sum_{\sigma \in S_4} (-1)^{\sigma}(-\partial_{\sigma(\bar{\beta})}\partial_{\sigma(\alpha)}h_{j\bar{k}}+ \sum h^{a\bar{b}}\partial_{\sigma(\alpha)}h_{j\bar{b}}\partial_{\sigma{\beta}}h_{a\bar{k}})\cdot\\
&\cdot(-\partial_{\sigma(\bar{\theta})}\partial_{\sigma(\gamma)}h_{k\bar{i}}+ \sum h^{a\bar{b}}\partial_{\sigma(\gamma)}h_{k\bar{b}}\partial_{\sigma{\theta}}h_{a\bar{i}})] dz^{\alpha} \wedge dz^{\beta} \wedge d\bar{z}^{\gamma} \wedge d\bar{z}^{\theta}
\end{align*}

behave in this case. It is clear that $\partial_{z_1}h_{ij}=0, \partial_{\bar{z_1}} h_{ij}=0$ but the derivatives from the other direction may not be zeroes (i.e. $\partial_z h_{ij}$ is not zero,  $z \in C$). If we look at the expression above it could be deduced that first derivatives are included in the expression in pairs $\sum h^{a\bar{b}}\partial_{\partial _{z_i}}h_{k\bar{b}}\partial_{\bar{z_j}}h_{a\bar{i}}$ (so that at least one of them is always zero as the connection is flat in one direction and hence the first derivative is zero) therefore the terms involving it go to zero and what lasts are the only expressions for the second derivatives.

If all the expressions with the first derivatives go to zero, we will get the same formula [$\ref{omega}$]. 
\begin{equation}
 \partial^2(h_{ij}) =\begin{pmatrix}
0 & 0  & \bar{a_{i,i}} & \bar{a_{i,j}} \\
0 & 0  & \bar{a_{j,i}} & \bar{a_{j,j}} \\
a_{i,i} & a_{i,j} &  * & * \\
a_{i,j} & a_{j,j} & * & *
\end{pmatrix}
\end{equation}
We denote by $*$ the elements of derivatives of metric $h_{ij}$ not in the mixed directions. Denote determinants of the matrices in the mixed directions as  
\begin{align*}
D_{ij}=\begin{pmatrix}
a_{i,i} & a_{i,j}\\
a_{i,j} & a_{j,j}
\end{pmatrix}
\end{align*}
and $\bar D_{ij}$ correspondingly.
In general, for both cases $E\to C_1 \times C_2$ and $E\to \mathcal{C}$   $det({\partial}^2(h_{ij}))=det(D_{ij})det(\bar D_{ij})$  (in the second case the additional minor doesn't influence on the expression of determinant by the Linear Algebra rules).

\begin{thm}
\label{MAIN}
The resulting formula for $E \to \mathcal{C}$ is

\begin{equation}
c_2(E)=\sum_i (\Omega^{4})_{ii}=
\sum_{i,j} det(\partial^{2}h_{ij})  dz^{1} \wedge d\bar{z}^{1} \wedge dz^{2}  \wedge d\bar{z}^{2}=\sum_{i,j} det(D_{ij})det(\bar D_{ij}) dz^{1} \wedge d\bar{z}^{1} \wedge dz^{2}  \wedge d\bar{z}^{2}  
\end{equation}
\end{thm}

Thus from the last theorem follows the stability result:
\begin{cor}

 We have for $E \to C_1 \times C_2$ and $E \to \mathcal{C}$
$$
c_2(E) >0
$$
\end{cor}

\section{Further directions: Stability of vector bundles on the algebraic surfaces}
We plan to investigate stability and the concrete formulas for the second chern class on a bigger class of complex surfaces, in particular on ones with the singularities.
The cornerstone hypothesis in this direction would be \\
\textbf{Hypothesis.} Every algebraic surface admits a map of smooth curves of the constant genus(i.e. a map of the equigeneric family to it).\\

\begin{rem}
If the hypothesis is indeed true, the general result for the positivity of $c_2(E)$ of a vector bundle $E \to X$ on a complex surface $X$ immediately follows from the [\ref{MAIN}].
\end{rem}

In \cite {B4} Buonerba and the first named author state the following theorems and facts which seems to be the closest to our current hypothesis:
\begin{thm}
The following classes $\mathscr{C}(k, n)$ are dominant:
\begin{enumerate}
    \item For $n= 3$,  smooth  threefolds  with a smooth  connected morphism  onto a smooth curve.
    \item For any $n$ and $k$ a finite field, projective varieties admitting a connected morphism onto a smooth curve, with only one singular fiber whose singular locus consists of one ordinary double point. pencils.
\end{enumerate}

\end{thm}
The main idea in \cite{B4} was to construct fibrations using Lefschetz pencils.\\
\textbf{Fact.} For  any  field k and  integer $n \geq 2$,  the  class  of  projective  varieties  admitting  a connected morphism onto $\mathbf{P}^1_k$, with isolated singular fibers whose singular locus consists of one ordinary double point, is dominant. Statement (1) of the Theorem is then an immediate consequence of the Brieskorn-Tyurina’s simultaneous resolution of surface ordinary double points, which in fact provides a simultaneous resolution of the fibers of the fibration induced by the Lefschetz pencil.

Since the curves in our family $u_C$ are equigeneric - meaning that the geometric genus is constant along the fibers - we have that $u^{norm}_C$ is again an equigeneric family, with smooth general member.

\begin{rem}(\cite{B4})

Let $X$ be a smooth surface with negative Kodaira dimension.  Then for $n$ sufficiently  big,  there  exists  a  smooth  proper  curve $C$ and  a  non-constant  morphism $f:C\to S^{k}_{n} \backslash S^{{k+1}}_{n}$.  Therefore, $X$ can be dominated by an equigeneric family of curves.
\end{rem}

\begin{thm}(\cite{B4},Brunella, Corlette and Simpson). Let $(X,\mathscr{F})$ be a smoothly foliated surface of general type. Then at least one of the following happens:
\begin{itemize}
    \item $X$ admits a smooth fibration $p:X \to C$ onto a smooth curve $C$ and $\mathscr{F}$ is tangent to $p$;
\item $\rho: \Gamma \to P SL_2(\mathbf{R})$ is rigid and integral,  there exists a quasiprojective  polydisk  quotient, Y,  and a natural morphism $X\to Y$such that $\mathscr{F}$ is induced by one of tautological codimension one foliations on $Y$.
\end{itemize}
\end{thm}

\section{Future directions: Calabi-Yau metric}
One of our future projects concerns the metric issues on the vector bundles over surfaces. We plan to find and write down an explicit formula for the Calabi-Yau metric on a vector bundle $E \to X$ over the complex surface $X$.

\bibliographystyle{amsplain}

\begin{thebibliography}{10}

\bibitem{NS} Narasimhan, M. S. and Seshadri, C. S., \textit{Stable and unitary vector bundles on a compact {R}iemann surface}, Ann. of Math. (2),\textbf{82} (1965), 540--567

\bibitem{HL} Huybrechts, Daniel and Lehn, Manfred  \textit{The geometry of moduli spaces of sheaves}, Cambridge Mathematical Library (Second),(2010)

\bibitem{B1} Bogomolov, F. A. \textit{Stable vector bundles on projective surfaces}, Rossi\u{\i}skaya Akademiya Nauk. Matematicheski\u{\i} Sbornik (185),\textbf{4} (1994), 3--26

\bibitem{B3} Bogomolov, F. A. \textit{Holomorphic tensors and vector bundles on projective
              manifolds}, Izv. Akad. Nauk SSSR Ser. Mat. (42),\textbf{6} (1978), 1227--1287, 1439

\bibitem{B4} Buonerba, Federico and Bogomolov, Fedor A. \textit{Dominant classes of projective varieties}, European Journal of Mathematics (4), (2018)


\end{thebibliography}

\end{document}